\documentclass[12pt,a4paper,reqno]{amsart}
\pdfoutput=1

\usepackage[utf8]{inputenc}
\usepackage[british]{babel}
\usepackage{amsmath,amsthm,amssymb,mathtools}
\mathtoolsset{centercolon}
\usepackage{hyperref,float,caption}
\usepackage{geometry}
\usepackage[nameinlink]{cleveref}
\usepackage{enumitem}
\usepackage[abbrev, msc-links, nobysame]{amsrefs}
\usepackage[presets={vec-cev}]{letterswitharrows}
\usepackage{tikz,esvect}
\usetikzlibrary{arrows}
\usetikzlibrary{fadings}
\tikzset{arc/.style = {->,> = stealth'}}
\usetikzlibrary{angles,quotes}
\hypersetup{
    colorlinks,
    linkcolor={red!60!black},
    citecolor={green!60!black},
    urlcolor={blue!60!black},
}

\def\N{\mathbb N}

\def\P{\mathcal P}
\def\Q{\mathcal Q}
\def\R{\mathcal R}

\def\F{\mathcal F}

\def\P{\mathcal P}
\def\E{\mathcal E}
\def\C{\mathcal C}
\def\eps{\varepsilon}

\newcommand{\braces}[1]{\lbrace #1 \rbrace}
\newcommand{\verts}[1]{\vert #1 \vert}

\newcommand{\tribe}{\mathcal{F}}

\newtheorem{theorem}{Theorem}[section]
\newtheorem{lemma}[theorem]{Lemma}

\newtheorem{proposition}[theorem]{Proposition}
\theoremstyle{definition}

\newtheorem{problem}[theorem]{Problem}
\theoremstyle{remark}
\newtheorem*{claim}{Claim}

\title{Ubiquity of Oriented Rays}
\author{Florian Gut
    \and
        Thilo Krill
    \and
        Florian Reich}

\address{Universit\"at Hamburg, Department of Mathematics, Bundesstrasse 55 (Geomatikum), 20146 Hamburg, Germany}
\email{\{florian.gut, thilo.krill, florian.reich\}@uni-hamburg.de}

\begin{document}

\begin{abstract}
     A digraph $H$ is called \emph{ubiquitous} if every digraph $D$ that contains $k$ vertex-disjoint copies of $H$ for every $k \in \mathbb{N}$ also contains infinitely many vertex-disjoint copies of $H$.
    We characterise which digraphs with rays as underlying undirected graphs are ubiquitous.
\end{abstract}

\maketitle

\section{Introduction}

A (di)graph $H$ is called $\trianglelefteq$-\emph{ubiquitous} for a binary (di)graph relation $\trianglelefteq$ if any (di)graph $G$ that contains $k$ disjoint copies of $H$ for every $k \in \N$ also contains infinitely many disjoint copies of $H$ with respect to $\trianglelefteq$.
Possible relations for $\trianglelefteq$ are e.g.\ the subgraph, topological minor or minor relation for graphs or the subdigraph relation for digraphs.

Halin started the investigation of ubiquity in graphs with his landmark result that rays are subgraph-ubiquitous in~\cite{halin1965}.
Andreae conjectured that every locally finite connected graph is minor-ubiquitous after studying minor-ubiquity in \cites{Andreae2002,Andreae2013}.
Noteworthy progress towards this conjecture was recently achieved by Bowler, Elbracht, Erde, Gollin, Heuer, Pitz and Teegen in a series of papers~\cites{BEEGHPT22, BEEGHPT18, BEEGHPT20}, in which they proved, among several other results, that all trees are topological-minor-ubiquitous.
Throughout the years several results proving and disproving the ubiquity of certain graphs have been published, including results concerning different notions of ubiquity as in \cites{bowler2013,KURKOFKA2021103326}.

\begin{figure}[h]
    \center
	\begin{tikzpicture}[scale=1.8]
		\foreach \x in {0,1,2,3,4,5} \foreach \y in {0,0.8} \draw[fill] (\x,\y) circle (1.2pt);
		\foreach \x in {0,1,2,3,4,5} \draw[thick] (\x,0) -- (\x,0.8);
		\draw[thick] (0,0) -- (5.5,0) node[right] {$\dots$};
	\end{tikzpicture}
	\caption{The comb, a graph that is not subgraph-ubiquitous.}\label{fig:comb}
\end{figure}
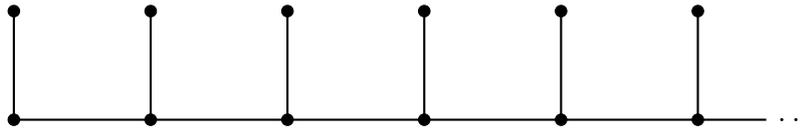

An example for a graph that is not subgraph-ubiquitous is the comb \cite{Andreae2002} (see \Cref{fig:comb}).
For graphs that are not topological-minor-ubiquitous, see~\cite{Andreae2013}.
Very recently, Carmesin provided an example of a locally finite graph that is not minor-ubiquitous in \cites{carmesin2022}.
We remark that this does not contradict Andreae's conjecture as the graph is not connected.

In \cite{bowler2013} Bowler, Carmesin and Pott first suggested the topic of ubiquity in digraphs by asking whether any digraph containing arbitrarily many edge-disjoint directed double rays also contains infinitely many of them.

We take on the quest of investigating ubiquity in digraphs.
We characterise which \emph{oriented rays}, i.e.\ digraphs whose underlying undirected graphs are rays, are ubiquitous regarding the subdigraph relation.
Whenever we write \emph{ubiquitous} without specifying the relation, we refer to the subdigraph relation.
Furthermore we call an arc of an oriented ray $R$ \emph{in-oriented} if it is directed towards the unique vertex of $R$ with undirected degree $1$ and \emph{out-oriented} otherwise.
Our main result reads as follows:

\begin{theorem}\label{theo:main_theo}
An oriented ray is ubiquitous if and only if either all but finitely many arcs are in-oriented or all but finitely many arcs are out-oriented.
\end{theorem}

In this paper we develop novel methods that enhance the common techniques in the field of ubiquity theory.
For the backward implication of \Cref{theo:main_theo} (see \Cref{section:positive}) we follow the proof for Halin's ray ubiquity result in undirected graphs which can be found in \cite{diestel2017}*{Theorem 8.2.5 (i)}.
To make this possible we require sets of arbitrarily many disjoint copies of an oriented ray that have an additional property, they must be forked.
The existence of such forked sets, which we prove in \Cref{lem:disjoint_initial_segments}, is our key contribution in the proof of the forward implication.

For the forward implication of \Cref{theo:main_theo} (see \Cref{section:negative}) we construct a counterexample from infinitely many disjoint copies of an oriented ray by identifying vertices.
In the proof of \Cref{theo:negative_result_no_period} we extend the common technique of identifying vertices by a recursive choice of the vertices which will be identified.

Let us present some open problems concerning the ubiquity of different digraphs.
We call digraphs, whose underlying undirected graphs are double rays, \emph{oriented double rays}.
\begin{problem}
Which oriented double rays are ubiquitous?
\end{problem}
A special case that turns out to be quite challenging is:
\begin{problem}
Is the consistently oriented double ray, i.e.\ every vertex has in-degree and out-degree 1, ubiquitous?
\end{problem}

More generally, one can investigate ubiquity of digraphs whose underlying undirected graphs are trees.
However, even the question which undirected trees are subgraph-ubiquitous is unsolved and only known for ubiquity with respect to weaker relations such as the topological minor relation.
Therefore it might be sensible also to discuss ubiquity of digraphs with respect to weaker relations such as butterfly minors.
Moreover, since proving or disproving the ubiquity of consistently oriented double rays is not easy, we propose to initially consider \emph{out-trees}, i.e.\ trees in which all arcs are oriented away from the root.
\begin{problem}
Which out-trees are ubiquitous concerning a fitting notion of ubiquity?
\end{problem}

\section{Preliminaries}\label{section:preliminaries}

For general graph theoretic notation we refer to~\cite{bang2008} and~\cite{diestel2017}.
From now on we write $\emph{ray}$ instead of oriented ray.
Similarly, a \emph{path} is a digraph whose underlying undirected graph is a path.
We call a path a \emph{dipath} if all its arcs are consistently oriented, i.e.\ each vertex has in- and out-degree at most $1$.
Moreover, an \emph{out-ray} is a ray in which all arcs are out-oriented.

For $n \in \N$ we denote $[n] := \braces{1, \dots , n}$.
A vertex of a ray which is incident with two outgoing or two incoming arcs is called a \emph{turn}.
A maximal dipath contained in a ray is called a \emph{phase}.
We call a phase of a ray \emph{in-oriented} if its arcs are in-oriented and \emph{out-oriented} otherwise.

A ray that has infinitely many in-oriented and infinitely many out-oriented arcs can be represented by a sequence of natural numbers where the $n$-th term of the sequence represents the length of the $n$-th phase.
We call this the \emph{representing sequence} of the ray.
The representing sequence is \emph{bounded} if there is $b \in \N$ such that all elements of this sequence are contained in $[b]$.
Otherwise the representing sequence is called \emph{unbounded}.

For a digraph $D$ and vertices $v,w\in D$, we write $d_D(v,w)$ for the distance between $v$ and $w$ in the underlying undirected graph.
Let $R\subseteq D$ be a ray.
For any $v \in R$ we write $vR$ for the tail of $R$ starting in the vertex $v$, and for any $a \in A(R)$ we write $aR$ for the tail of $R$ starting with the arc $a$.
For $w \in D$ we say that $v \in R$ lies \emph{beyond $w$ on $R$} if $w\notin V(vR)$.
Two rays $R_1, R_2$ \emph{traverse an arc $uv \in A(R_1)\cap A(R_2)$ in the same direction} if either $v$ lies beyond $u$ on $R_1$ and $R_2$ or $u$ lies beyond $v$ on $R_1$ and $R_2$.
Otherwise $R_1$ and $R_2$ \emph{traverse $uv$ in opposite directions}.

Let $\Dv$ be the digraph obtained from $D$ by changing the orientation of every arc.
To warm up with the definition of ubiquity, we prove the following simple lemma:
\begin{lemma}\label{lem:change_orientation}
A digraph $H$ is ubiquitous if and only if $\Hv$ is ubiquitous.
\end{lemma}
\begin{proof}
It suffices to show that $\Hv$ is ubiquitous if $H$ is ubiquitous.
Let $D$ be any digraph containing arbitrarily many disjoint copies of $\Hv$.
Hence $\Dv$ contains arbitrarily many disjoint copies of $\cev{\Hv}=H$.
Then $\Dv$ also contains infinitely many disjoint copies of $H$ since $H$ is ubiquitous.
Therefore $\cev{\Dv}=D$ contains infinitely many disjoint copies of $\Hv$, which proves that $\Hv$ is ubiquitous.
\end{proof}
Corresponding to \cite{BEEGHPT22}*{Definition 5.1}, for digraphs $D$ and $R$ we call a collection $\tribe$ of finite sets of disjoint copies of $R$ in $D$ an \emph{$R$-tribe in $D$}.
If $R$ is clear from the context we may also just say \emph{tribe} instead, similarly for the containing graph $D$.
Furthermore, for an $R$-tribe $\tribe$ in $D$ we call $F \in \tribe$ a \emph{layer of $\tribe$}, any element of $F$ a \emph{member of $\tribe$} and say that $\tribe$ is \emph{thick} if for each $n \in \N$ there is a layer $F$ of $\tribe$ with $\verts{F} \geq n$.
Note that if $D$ contains arbitrarily many disjoint copies of $R$, then $D$ contains a thick $R$-tribe.
A tribe $\tribe'$ in $D$ is an \emph{($R$-)subtribe}\footnote{
Note that this definition of subtribe corresponds to the notion of flat subtribe in \cite{BEEGHPT22}. The definition of subtribe in \cite{BEEGHPT22} however is different and more general.} of an $R$-tribe $\tribe$ in $D$ if every layer of $\tribe'$ is a subset of a layer of $\tribe$.

In this paper, whenever we consider a copy $R'$ of a digraph $R$ we implicitly fix an isomorphism $\varphi\colon R \to R'$ and for a subdigraph $\hat{R} \subseteq R$ we write in short $\hat{R}'$ for $\varphi(\hat{R})$.
With this, we say that an $R$-tribe $\tribe$ is \emph{forked at} $\hat{R}$ if $\hat{R}' \cap {R}'' = \emptyset$ for any two distinct members $R', R''$ of $\tribe$.

\section{Positive results}\label{section:positive}

In this section, we prove

\begin{theorem}\label{theo:positive_result}
A ray is ubiquitous if all but finitely many arcs are in-oriented or all but finitely many arcs are out-oriented.
\end{theorem}

The proof of \Cref{theo:positive_result} will be an easy consequence of \Cref{theo:ubiquity_with_starting_set} together with \Cref{lem:disjoint_initial_segments} below.
The following \lcnamecref{theo:ubiquity_with_starting_set} is a variant of Halin's ray ubiquity result for digraphs with an additional restriction on the start vertices of the rays.
The proof given here is derived from the proof of Halin's result in \cite{diestel2017}*{Theorem 8.2.5 (i)}.

\begin{theorem}\label{theo:ubiquity_with_starting_set}
    Let $D$ be a digraph and $R$ an out-oriented ray.
    If there exists a thick $R$-tribe $F$ in $D$ and $X \subseteq V(D)$ such that each member of F has its first vertex in $X$, then there are infinitely many disjoint out-rays in $D$ whose first vertices are contained in $X$.
\end{theorem}

\begin{proof}
    We will recursively fix for every $n \in \N^{+}$ a set $\R^n = \braces{R^n_1,\dots, R^n_n}$ of pairwise disjoint out-rays in $D$ and a set of vertices $\braces{u^n_1, \dots, u^n_n}$ such that for every $k \in [n]$:
    \begin{itemize}
        \item $R^n_k$ has its first vertex in $X$,
        \item $u^n_k \in R^n_k $, and
        \item $R^{n}_ku^{n}_k \subsetneq R^{n+1}_ku^{n+1}_k$.
    \end{itemize}
    Then $\braces{\bigcup_{n\geq k} R^n_ku^n_k \colon k \in \N}$ is an infinite set of pairwise disjoint out-rays where each of the rays has its first vertex in $X$.
    
    For $n=1$ we pick one ray $R^1_1 \in \bigcup \tribe$, set $\R^1 := \braces{R^1_1}$ and pick $u^1_1 \in R^1_1$ arbitrarily.
    This satisfies the required properties.

    Now let $\ell \geq 1$ and suppose that for all $i \in [\ell]$ there are sets $\R^i$ and $\braces{u^i_1, \dots, u^i_i}$ subject to the conditions above.
    Consider a layer $F$ of $\tribe$ of size at least $\verts{\bigcup_{i \in [\ell]} R^\ell_i u^\ell_i } + \ell^2 +1$.
    First, we delete from $F$ every ray that meets a path $R^\ell_i u^\ell_i$ for some $i \in [\ell]$.
    Then there are still at least $\ell^2+1$ rays left in $F$.
    Next, we repeatedly check whether there is a ray $R^\ell_i \in \R^\ell$ for which $R^{\ell+1}_i$ has not yet been defined and that meets at most $\ell$ of the remaining elements in $F$.
    If that is the case, we set $R^{\ell+1}_i := R^\ell_i$, choose a vertex $u^{\ell+1}_i$ beyond $u^{\ell}_i$ on $ R^\ell_i$ arbitrarily, and delete the at most $\ell$ many rays from $F$ that have non-empty intersection with $R^{\ell+1}_i$.
    Suppose that after $m \leq \ell$ many steps, every ray in $\R^\ell$ meets either none or more than $\ell$ many rays from the reduced $F$, which we will refer to as $F^\prime$.
    
    Consider the $(\ell-m)$-sized subset $J\subseteq [\ell]$ containing all $j\in [\ell]$ for which $R^{\ell+1}_j$ has not yet been defined.
    Then any ray $R^\ell_j$ with $j\in J$ meets more than $\ell$ rays from $F^\prime$.
    We deleted at most $\verts{\bigcup_{i \in [\ell]} R^\ell_i u^\ell_i }$ rays from $F$ in the first step and at most $m \ell$ in the second step, thus $F^\prime$ has size at least
    \[
        \verts{\bigcup_{i \in [\ell]} R^\ell_i u^\ell_i } + \ell^2 +1 - \verts{\bigcup_{i \in [\ell]} R^\ell_i u^\ell_i } - m\ell = (\ell-m)\ell+1.
    \]
    
    For any ray $R^\ell_j$ with $j\in J$ we fix the vertex $c_j\in R^\ell_j$ which is the first intersection of $R^\ell_j$ with the $\ell$-th ray from $F^\prime$ that it meets.
    Note that $c_j$ lies beyond $u^{\ell}_j$ on $R^{\ell}_j$.
    Then $\bigcup_{j\in J} R^\ell_j c_j$ meets at most $\verts{J}\ell=(\ell-m)\ell$ rays from $F^\prime$.
    Therefore, there is at least one ray left in $F^\prime$ that is disjoint from $\bigcup_{j\in J} R^\ell_j c_j$ and we pick this ray as $R^{\ell+1}_{\ell+1}$.
    We choose an arbitrary vertex $u_{\ell+1}^{\ell+1} \in R_{\ell+1}^{\ell+1}$, define $F^\ast:=F'\setminus\{R_{\ell+1}^{\ell+1}\}$, and write $F^\ast=\{S_i:i\in I\}$ for a suitable index set $I$.

    Now for any ray $S_i\in F^\ast$ we choose a vertex $w_i$ that lies beyond all vertices of $\bigcup_{j \in J} u^\ell_j R^\ell_j c_j$ on $S_i$.
    Consider the finite subdigraph
    \[
        H := \bigcup_{j \in J} u^\ell_j R^\ell_j c_j \cup \bigcup_{i \in I} S_i w_i
    \]
    of $D$.
    Additionally, we define $U:= \braces{u^\ell_j \colon j \in J}$ and $W:= \braces{w_i \colon i \in I}$.
    We show that for any set $Z\subseteq V(H)$ of fewer than $\ell-m$ vertices there is a $U$--$W$ dipath in $H - Z$.
    Indeed, $Z$ misses at least one dipath of the form $u^\ell_j R^\ell_j c_j$ and since there are $\ell \geq \ell-m$ many paths $S_i w_i$ with $u^\ell_j R^\ell_j c_j\cap S_i w_i \neq \emptyset$, at least one such path $S_i w_i$ avoids $Z$.
    Let $v_j$ be the first vertex on $u^\ell_j R^\ell_j c_j$ which lies on $S_i w_i$; then $u^\ell_j R^\ell_j v_j S_i w_i$ is a dipath from $u^\ell_j$ to $w_i$:
    firstly, by the choice of $v_j$ the underlying undirected graph of $u^\ell_j R^\ell_j v_j S_i w_i$ clearly is a path.
    Secondly, the dipath $u^\ell_j R^\ell_j v_j$ is directed from $u^\ell_j$ to $v_j$ since $v_j \in u^\ell_j R^\ell_j$.
    Lastly, the dipath $v_j S_i w_i$ is directed from $v_j$ to $w_i$ since $v_j$ lies in $S_i w_i$.
    Thus by Menger's theorem~\cite{bang2008}*{Theorem~7.3.1}, there is a set $\P$ of $\ell - m = \verts{J}$ pairwise disjoint $U$--$W$ dipaths in $H$.

    For all $j\in J$, we write $P_j$ for the dipath in $\P$ starting at $u^\ell_j$.
    Let $h\colon J \to I$ such that $w_{h(j)}$ is the endvertex of $P_j$ in $W$.
    Now we define
    \[
        R^{\ell+1}_{j} := R^\ell_j u^\ell_j P_j w_{h(j)} S_{h(j)}
    \]
    and $u^{\ell+1}_j := w_{h(j)}$, which clearly fulfils the required properties.
\end{proof}

\begin{lemma}\label{lem:disjoint_initial_segments}
    Let $D$ and $H$ be digraphs and let $\hat{H} \subseteq H $ a finite subdigraph.
    If there exists a thick $H$-tribe $\E$ in $D$, then there is a thick $H$-subtribe $\F$ of $\E$ in $D$ that is forked at $\hat{H}$.
\end{lemma}

\begin{proof}
    For all $n \in \N$, we recursively define a subset $F_n$ of a layer of $\E$ containing at least $n$ disjoint copies of $H$ in $D$ and a thick subtribe $\E_{n}$ of $\E$ such that
    \begin{enumerate}[label=(\roman*)]
        \item\label{item:F_is_forked} the $H$-tribe $\F_n := \{F_0,\dots,F_n\}$ is forked at $\hat{H}$,
        \item\label{item:two_families_are_forked} for each $H_1 \in \bigcup \E_{n}$ and each $H_2 \in \bigcup \F_{n}$ the digraph $\hat{H}_1$ is disjoint from $H_2$ and the digraph $\hat{H}_2$ is disjoint from $H_1$.
    \end{enumerate}
    In the end, $\{F_n : n \in \N\}$ will be a thick $H$-subtribe of $\E$ satisfying the \namecref{lem:disjoint_initial_segments}.
    For the first step we set $F_0 := \emptyset$ and $\E_0 := \E$.
    Now suppose that $\F_{n-1}$ and $\E_{n-1}$ are already defined.
    Set $h := |\hat{H}|$ and choose a layer $L$ from $\E_{n-1}$ of size at least $h + n$.
    We will choose $F_n$ as an $n$-element subset of $L$.
    Then $\F_n$ will be forked at $\hat{H}$ since \labelcref{item:F_is_forked} and \labelcref{item:two_families_are_forked} hold for $\E_{n -1}$ and $\F_{n -1}$.
    Our task is to find a suitable subset $F_n$ of $L$ and a thick subtribe $\E_n$ of $\E$ such that $\bigcup\F_n$ and $\bigcup \E_n$ satisfy \labelcref{item:two_families_are_forked}.
    We begin by deleting from each layer $M \neq L$ of $\E_{n-1}$ any element that has non-empty intersection with some $H^\prime \in L$ in its subdigraph $\hat{H}^\prime$.
    Note that for every digraph $H^\prime\in L$ there are at most $\verts{\hat{H}^\prime}=h$ many digraphs from $M$ which meet $\hat{H}^\prime$.
    Therefore we delete from every layer of $\E_{n-1}$ at most $h\cdot\verts{L}$ elements and the resulting subtribe $\C$ of $\E_{n-1}$ is still a thick tribe in $D$.
    
    \begin{claim}
        For every $j \in \N$ there is a subset $L_j \subseteq L$ with $\verts{L_j} = n$ and a subset $C_j$ with $\verts{C_j} \geq j$ of a layer of $\C$ such that for any $H_1 \in L_j$ and any $H_2 \in C_j$ the digraph $H_1$ is disjoint from $\hat{H}_2$ and $H_2$ is disjoint from $\hat{H}_1$.
    \end{claim}
    \begin{proof}[Proof of the claim]
        Let $j \in \N$ and $C$ a layer of $\C$ of size at least $j\binom{\verts{L}}{n}$.
        By the construction of $\C$, we only need to find sets $L_j \subseteq L$ and $C_j \subseteq C$ such that no $H_1 \in L_j$ meets any $H_2 \in C_j$ in its subdigraph $\hat{H}_2$.
        For every $H^\prime \in C$, at most $\verts{\hat{H}^\prime}=h$ elements of $L$ meet $\hat{H}^\prime$.
        Since $ \verts{L} \geq h+n$, we can choose for every $H^\prime \in C$ a subset of $n$ elements of $L$ such that each of these does not meet $\hat{H}^\prime$.
        This defines a map $\alpha \colon C \rightarrow \mathfrak{L}:= \braces{L^\prime \subseteq L \colon \verts{L^\prime} = n}$.
        Since $\verts{C}\geq j \binom{\verts{L}}{n}=j\verts{\mathfrak{L}}$, there is a set $L_j\in\mathfrak{L}$ with $\verts{\alpha^{-1}(L_j)}\geq j$ by pigeon hole principle.
        Then $L_j$ and $C_j:=\alpha^{-1}(L_j)$ are as desired.
    \end{proof}
    Since $L$ has only finitely many subsets, there is an infinite strictly increasing sequence $(j_k)_{k\in\N}$ in $\N$ such that the sets $L_{j_k}$ coincide for all $k\in\N$.
    We choose this as the set $F_n$.
    By the claim, $\E_n := \braces{C_{j_k} \colon k \in \N}$ is a thick subtribe of $\E_{n-1}$ satisfying \ref{item:two_families_are_forked}.
    This concludes the proof.
\end{proof}

\begin{proof}[Proof of \Cref{theo:positive_result}]
Let $R$ be a ray where all but finitely many arcs are oriented the same way. Let $D$ be a digraph and assume that $D$ contains arbitrarily many
disjoint copies of $R$, then $D$ contains a thick $R$-tribe $\mathcal{E}$.
We show that $D$ contains infinitely many copies of $R$.
By \Cref{lem:change_orientation} we may assume that all but finitely many arcs of $R$ are out-oriented.

Let $\hat{R}$ be the (connected) subdigraph of $R$ that consists precisely of all finite phases of $R$.
By \Cref{lem:disjoint_initial_segments}, there is a thick subtribe $\tribe$ of $\E$ that is forked at $\hat{R}$.
Consider the set $X$ which contains for any $R'\in\bigcup \tribe$ the first vertex of the out-ray $R'-\hat{R'}$, and the set $Y:=\bigcup_{R'\in\bigcup\tribe}V(\hat{R}')$.
By deleting $\hat{R}'$ from each member $R'$ of $\tribe$, we obtain a thick $(R-\hat{R})$-tribe in $D-Y$.
Hence, by \Cref{theo:ubiquity_with_starting_set} there exists an infinite family $(R_i)_{i\in\N}$ of disjoint out-rays in $D-Y$ such that each $R_i$ starts in a vertex $r_i\in X$.
By definition of $X$, for all $i\in\N$ there is a member $S_i$ of $\tribe$ such that $r_i$ is the first vertex of $S_i-\hat{S_i}$.
Note that $\hat{S_i}$ and $\hat{S_j}$ are disjoint for $i\neq j$ since $\tribe$ is forked at $\hat{R}$.
Finally, by combining each initial segment $\hat{S_i}$ with the out-ray $R_i$, we obtain infinitely many disjoint copies of $R$ in $D$.
\end{proof}

\section{Negative results}\label{section:negative}

In this section, we construct for every oriented ray $R$ with infinitely many turns a digraph $D$ that contains arbitrarily but not infinitely many disjoint copies of $R$.
The construction of $D$ will differ depending on whether the representing sequence of $R$ is bounded (see \Cref{theo:negative_result_bounded_sequence}) or unbounded (see \Cref{theo:negative_result_no_period}).
However, the basic framework for the construction of $D$ will be the same in both proofs as follows:

Let $(R(n,m))_{(n,m)\in I}$ be a family of pairwise disjoint copies of $R$, where
\[I:=\{(n,m) \in \N^2 \colon n \leq m\}.\]
We let
\[D_{-1}:= \bigcup_{(n,m)\in I} R(n,m),\]
\[J := \{ ((n^0,m^0),(n^1,m^1)) \in I^2 \colon m^0 < m^1 \}\]
and fix an arbitrary sequence $((n^0_i,m^0_i),(n^1_i,m^1_i))_{i\in\N}$ in $J$ which contains every element of $J$ infinitely often.
Further, let $(g^0_i,g^1_i)_{i\in\N}$ be a sequence of pairwise disjoint pairs of vertices of $D_{-1}$ with $g^0_i \in R(n^0_i,m^0_i)$ and $g^1_i \in R(n^1_i,m^1_i)$ for all $i \in \N$ and let
$D$ be the digraph obtained from $D_{-1}$ by identifying $g^0_i$ and $g^1_i$ for any $i \in \N$.

We can think of $D$ as having the same arc set as $D_{-1}$ and refer to the ray in $D$ with arc set $E(R(n,m))$ as $R(n,m)$ for any $(n,m) \in I$.

\begin{proposition}\label{prop:following_tail}
Every digraph $D$, constructed as above, contains $k$ disjoint copies of $R$ for all $k \in \N$.
Moreover, if any copy of $R$ in $D$ has a tail in $R(n,m)$ for some $(n,m) \in I$, then $D$ does not contain infinitely many disjoint copies of $R$, and particular, $R$ is non-ubiquitous.
\end{proposition}

\begin{proof}
By the choice of $I$ and $J$, we have identified infinitely many vertices of $R{(n,m)}$ and $R{(n',m')}$ if $m \neq m'$ and none otherwise in the construction of $D$.
Hence $D$ contains arbitrarily many disjoint copies of $R$ as the rays $R(0,m),R(1,m),\ldots ,R(m,m)$ are disjoint for all $m\in\N$.
	
Suppose for a contradiction that there is a family $\R$ of infinitely many disjoint copies of $R$ in $D$.
For any copy $R'$ of $R$ in $\mathcal{R}$, there is by assumption $(n,m) \in I$ such that a tail of $R'$ coincides with a tail of $R(n,m)$.
Since infinitely vertices of two rays $R{(n,m)}$ and $R{(n',m')}$ are identified if $m \neq m'$, there is a fixed $m^* \in \N$ such that each ray in $\mathcal{R}$ has a tail identical with a tail of $R(n,m^*)$ for some $n \in \N$.
But $n \leq m^*$ by the definition of $I$.
So tails of the rays in $\mathcal{R}$ are contained in finitely many rays, which contradicts that they are disjoint.
\end{proof}
In the proofs of \Cref{theo:negative_result_bounded_sequence,theo:negative_result_no_period} we will use different strategies to choose the sequences $(g^0_i,g^1_i)_{i\in\N}$ such that any copy of $R$ in $D$ has a tail in $R(n,m)$ for some $(n,m) \in I$.
Then $D$ contains arbitrarily but not infinitely many copies of $R$ by \Cref{prop:following_tail} as desired.

\begin{theorem}\label{theo:negative_result_bounded_sequence}
All rays with a bounded representing sequence are non-ubiquitous.
\end{theorem}

\begin{proof}
Let $R$ be an arbitrary ray with a bounded representing sequence.
Let $c$ be the largest natural number that occurs infinitely often in its representing sequence.
By \Cref{lem:change_orientation}, we may assume that infinitely many phases of length $c$ in $R$ are out-oriented.
Further, let $I, J$, $((n^0_i,m^0_i),(n^1_i,m^1_i))_{i\in\N}$ and $D_{-1}$ be as defined in the beginning of this section.

Now we define a sequence $(g^0_i,g^1_i)_{i\in\N}$ of pairwise disjoint pairs of vertices of $D_{-1}$ recursively with $g^0_i \in R(n^0_i,m^0_i)$ and $g^1_i \in R(n^1_i,m^1_i)$ for all $i \in \N$.
If $(g^0_j,g^1_j)$ has been defined for all $j < i$, we pick for $\varepsilon \in \{0,1\}$ the vertex $g^\varepsilon_i$ beyond all vertices $g^0_0,g^1_0,\dots, g^0_{i-1},g^1_{i-1}$ on $R(n^\varepsilon_i,m^\varepsilon_i)$ with the following properties (see \Cref{figure:glueing_bounded_sequence}):
\begin{enumerate}[label=(\roman*)]
    \item\label{item:glueing_prop_1} $g^1_i$ is a turn in $R(n^1_i,m^1_i)$ at the start of an out-oriented phase of length $c$, and
    \item\label{item:glueing_prop_2} $g^0_i$ is a turn in $R(n^0_i,m^0_i)$ at the end of an out-oriented phase of length $c$ with the property that $|R(n^0_i,m^0_i)g^0_i|> |R(n^1_i,m^1_i)g^1_i|$.
\end{enumerate}
This is possible since $R(n^0_i,m^0_i)$ and $R(n^1_i,m^1_i)$ contain infinitely many out-oriented phases of length $c$.
Let $D$ be the digraph constructed from $D_{-1}$ and $(g^0_i,g^1_i)_{i\in\N}$ as defined in the beginning of this section. We write $g_i$ for the vertex $g_i^0=g_i^1$ in $D$.

\begin{figure}[ht]
    \center
	\begin{tikzpicture}

		\draw[arc] (0.9,0) to (0.1,0);
		\draw[arc] (1.1,0) to (1.9,0);
		\draw[arc] (2.1,0) to (2.9,0);
		\draw[arc] (3.1,0) to (3.9,0);
		\draw[arc] (4.9,0) to (4.1,0);
		
		\draw[dotted] (-0.5,0) to (-0.1,0);
		\draw[arc] (5.3,0) to (5.1,0);  
	    \draw[path fading=east] (5.1,0) to (6,0);

		\draw[arc] (4,-0.1) to (4,-0.9);
		\draw[arc] (4,0.1) to (4,0.9);
		\draw[arc] (4,1.1) to (4,1.9);
		\draw[arc] (4,2.1) to (4,2.9);
		\draw[arc] (4,3.9) to (4,3.1);
  
		\draw[arc] (4,4.3) to (4,4.1);
		\draw[path fading=north] (4,5) to (4,4.1);
		\draw[dotted] (4,-1.1) to (4,-1.5);

		\draw (-0.6,0) node {\huge.};
		\draw (4,-1.6) node {\huge.};
		\draw (0,0) node {\huge.};
		\draw[red] (1,0) node {\huge.};
		\draw (2,0) node {\huge.};
		\draw (3,0) node {\huge.};
		\draw (5,0) node {\huge.};
		\draw[blue] (4,0) node {\huge.};
		\draw (4,-1) node {\huge.};
		
		\draw (4,2) node {\huge.};
		\draw (4,1) node {\huge.};
		\draw[red] (4,3) node {\huge.};
		\draw (4,4) node {\huge.};
		\draw[blue] (4.5,-0.5) node {$g_i$};
		\draw (-2,0) node {$R(n^0_i,m^0_i)$};
		\draw (4,-2) node {$R(n^1_i,m^1_i)$};
		
	\end{tikzpicture}
	\caption{Example of a vertex $g^0_i = g^1_i = g_i$ in $D$ for $c=3$.}
	\label{figure:glueing_bounded_sequence}
\end{figure}
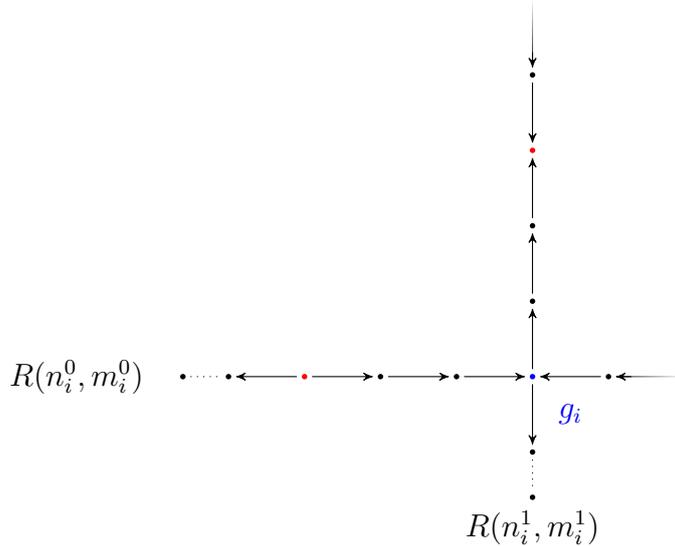

By \Cref{prop:following_tail}, it suffices to prove that for any copy $R'$ of $R$ in $D$ there is $(n,m) \in I$ such that a tail of $R'$ coincides with a tail of $R(n,m)$.
Begin by fixing an arbitrary copy $R'$ of $R$ in $D$.
Note that, by the construction of
$D$, $R'$ consists of (possibly infinite) segments contained in rays $R(n,m)$ for $(n,m) \in I$, and $R'$ can switch from $R(n,m)$ to $R(n',m')$ with $(n,m) \neq (n',m')$ only at some identification vertex $g_i$.
Since in the construction of $D$ we only identified turns and $c$ is the largest number which occurs infinitely often in the representing sequence of $R$, there is a tail $R''$ of $R'$ whose representing sequence contains only numbers up to $c$ and whose initial vertex is a turn of some ray $R(n^*,m^*)$ for $(n^*,m^*) \in I$. Thus each phase of any $R(n',m')$ for $(n',m') \in I$ is either completely traversed by $R''$ or all arcs of this phase are avoided by $R''$.

Let $i \in \N$ be arbitrary.
By properties \ref{item:glueing_prop_1} and \ref{item:glueing_prop_2}, any ray in $D$ traversing both the phase of $R(n^0_i, m^0_i) g_i $ incident with $g_i$ and an arc of $R(n^1_i,m^1_i)$ incident with $g_i$ contains a phase of length $>c$ (see \Cref{figure:glueing_bounded_sequence}).
Clearly, the same holds for a ray traversing both, the phase of $g_i R(n^1_i, m^1_i) $ incident with $g_i$ and an arc of $R(n^0_i,m^0_i)$ incident with $g_i$.
Recall from the construction of $D$ that the vertex $g_i$ has degree four.
Thus if $R''$ contains $g_i$ as an inner vertex, exactly one of the following properties holds:
\begin{enumerate}[label=(\arabic*)]
    \item both arcs of $R''$ incident with $g_i$ are contained in $R(n_i^0,m_i^0)$, or
    \item both arcs of $R''$ incident with $g_i$ are contained in $R(n_i^1,m_i^1)$, or
    \item one arc of $R''$ incident with $g_i$ is contained in $R(n^1_i,m^1_i) g_i$ and one is contained in $g_i R(n^0_i, m^0_i)$. \label{enum:properties_of_rays_3}
\end{enumerate}

This feature restricts in which way $R''$ is embedded into $D$.
First we observe:
\begin{claim} \label{clm:same_direction}
For any $(n,m) \in I$ and any arc $a \in A(R(n,m)) \cap A(R'')$, the rays $R(n,m)$ and $R''$ traverse $a$ in the same direction.
\end{claim}
\begin{proof}[Proof of the claim]
Suppose this does not hold.
Pick $(n,m)\in I$ and an arc $a \in A(R(n,m)) \cap A(R'')$ with $|R(n,m)a|$ minimal such that $a$ contradicts this property (see \Cref{fig:same_direction}).
Let $u, v$ be the endvertices of $a$ such that $|R(n,m)u|<|R(n,m)v|$ applies.
By minimality of $a$, the other arc of $R''$ incident with $u$ is not contained in $R(n,m)u$.
Therefore $u = g_i$ for some $i \in \N$, i.e.\ there is $(n',m')\in I$ with $m \neq m'$ such that $u \in V(R(n,m)) \cap V(R(n',m'))$.
Then \labelcref{enum:properties_of_rays_3} holds for $R''$ at the vertex $u$.
As $a \in A(uR(n,m))$, the inequality $m < m'$ holds by definition of $J$ and the other arc incident with $u$ in $R''$ is contained in $R(n',m')u$.
The rays $R''$ and $R(n',m')$ traverse this arc in opposite directions, but $|R(n',m')u|<|R(n,m)u|$ holds by property \ref{item:glueing_prop_2} of the construction.
This contradicts the minimality of $\vert R(n,m)a \vert$.
\end{proof}

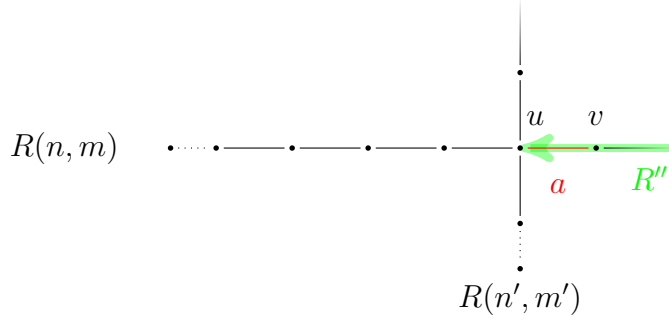
\begin{figure}[ht]
    \center
	\begin{tikzpicture}

        \draw[arc,line width=3pt,opacity=0.4,color=green] (6,0) to (4,0);

		\draw[] (0.9,0) to (0.1,0);
		\draw[] (1.1,0) to (1.9,0);
		\draw[] (2.1,0) to (2.9,0);
		\draw[] (3.1,0) to (3.9,0);
		\draw[red] (4.9,0) to (4.1,0);
		
		\draw[dotted] (-0.5,0) to (-0.1,0);
		\draw[path fading=east] (5.1,0) to (6,0);

		\draw[] (4,-0.1) to (4,-0.9);
		\draw[] (4,0.1) to (4,0.9);
	
		\draw[path fading=north] (4,2) to (4,1.1);
		\draw[dotted] (4,-1.1) to (4,-1.5);

		\draw (-0.6,0) node {\huge.};
		\draw (4,-1.6) node {\huge.};
		\draw (0,0) node {\huge.};
		\draw[] (1,0) node {\huge.};
		\draw (2,0) node {\huge.};
		\draw (3,0) node {\huge.};
		\draw (5,0) node {\huge.};
		\draw[] (4,0) node {\huge.};
		\draw (4,-1) node {\huge.};

		\draw (4,1) node {\huge.};
		\draw[] (4.2,0.4) node {$u$};
		\draw[] (5.,0.4) node {$v$};
		\draw[red] (4.5,-0.5) node {$a$};
		\draw (-2,0) node {$R(n,m)$};
		\draw (4,-2) node {$R(n',m')$};
  		\draw (5.7,-0.4) node {\textcolor{green}{$R''$}};
		
	\end{tikzpicture}
	\caption{The arc $a$ and the ray $R''$ in the proof of the claim for \Cref{theo:negative_result_bounded_sequence}.}
	\label{fig:same_direction}
\end{figure}

Now let $m \in \N$ be the smallest number such that there is $(n,m) \in I$ with $A(R'') \cap A(R{(n,m)}) \neq \emptyset$ and let $a$ be an element of this intersection.
We prove that $aR''$ coincides with $aR(n,m)$.
Suppose not and let $(n', m') \neq (n,m) \in I$ and $u \in V(R(n,m)) \cap V(R(n',m'))$ such that $u$ is incident with the first arc $b$ of $aR''$ not contained in $A(R(n,m))$.
Property \labelcref{enum:properties_of_rays_3} applies to $u$.
By minimality of $m$ we have $m<m'$ and thus $b \in A(R(n',m')u)$ by \labelcref{enum:properties_of_rays_3}.
Therefore $aR''$ and $R(n',m')$ traverse the arc $b$ in opposite directions, which contradicts the claim above.
Thus $R'$ has a tail that coincides with a tail of $R(n,m)$.
This completes the proof.
\end{proof}

It is left to prove \Cref{theo:negative_result_no_period}, for which we need the following Lemma:

\begin{lemma}\label{lem:unbounded_means_non-isomorphic}
Let $R$ be a ray with an unbounded representing sequence.
Then the tails $vR$ and $wR$ are non-isomorphic for all $v \neq w \in R$.
\end{lemma}

\begin{proof}
Let $R$ be a ray with an unbounded representing sequence.
Suppose that there are $v\neq w \in R$ such that $w$ lies beyond $v$ on $R$ and there is an isomorphism $\varphi \colon vR \rightarrow wR$. Clearly, we have $\varphi(v)=w$ and the paths $\varphi^n(v)R\varphi^{n+1}(v)$ are isomorphic for all $n\in\N$.
Therefore, the representing sequence of $vR$ is periodic.
Thus the representing sequence of $R$ is bounded, a contradiction.
\end{proof}

\begin{theorem}\label{theo:negative_result_no_period}
All rays with an unbounded representing sequence are non-ubiquitous.
\end{theorem}
\begin{proof}
Let $R$ be an arbitrary ray with an unbounded representing sequence.
Further, let $I, J, ((n^0_i,m^0_i),(n^1_i,m^1_i))_{i\in\N}$ and $D_{-1}$ be as defined in the beginning of this section.

We define a sequence $(g^0_i,g^1_i)_{i\in\N}$ of pairwise disjoint pairs of vertices of $D_{-1}$ recursively with $g^0_i \in R(n^0_i,m^0_i)$ and $g^1_i \in R(n^1_i,m^1_i)$ for all $i \in \N$.
We fix an enumeration $\{v_0,v_1,\ldots\}$ of $V(D_{-1})$.
Denote by $D_i$ the digraph obtained from $D_{-1}$ by identifying the vertices $g^0_\ell$ and $g^1_\ell$ for all $\ell\leq i$ and write $g_\ell$ for the vertex $g^0_\ell=g^1_\ell$.
When a vertex $v_j$ is identified with a vertex $v_k$ in this process, we call the new identification vertex both $v_j$ and $v_k$.
We will make sure that the following holds for all $i\in\N$:

\begin{enumerate}[label=(\roman*)]
\item\label{item:length_of_path_to_g} Let $k\leq i$ and let $S$ be a $v_k$--$g_i$ path in $D_i$ which is isomorphic to an initial segment of $R$.
Then $|S|=|R(n^0_i,m^0_i)g_i|$ or $|S|=|R(n^1_i,m^1_i)g_i|$.
\item\label{item:distance_to_endvertices} $g_i$ is not a turn of $R(n^0_i,m^0_i)$ or $R(n^1_i,m^1_i)$ and thus $g_i$ lies in the interior of phases $M^0$ of $R(n^0_i,m^0_i)$ and $M^1$ of $R(n^1_i,m^1_i)$.
Let $t^\eps_0,t^\eps_1$ be the two turns which are endvertices of $M^\eps$. Then the four numbers $d_{R(n^\eps_i,m^\eps_i)}(t^\eps_\delta,g_i)$ for $\delta,\eps\in\{0,1\}$ are pairwise distinct. Furthermore, $M^0$ and $M^1$ do not contain any $g_j$ with $j\neq i$.
\end{enumerate}

Let us first derive from the existence of a sequence $(g^0_i,g^1_i)_{i\in\N}$ as above that the digraph $D$, constructed from $D_{-1}$ and $(g^0_i,g^1_i)_{i\in\N}$ as in the beginning of this section, has the property that any copy of $R$ in $D$ has a tail that is contained in some $R(n,m)$.
Then \Cref{prop:following_tail} ensures that $D$ contains arbitrarily but not infinitely many disjoint copies of $R$.

We know that $R'$ traverses infinitely many vertices $g_j$ with $j\in\N$ (but $R'$ does not necessarily swap from one ray of the form $R(n,m)$ to another at every $g_j$) because every ray $R(n,m)$ is glued together with other rays at infinitely many vertices in $D$.
Suppose that $R'$ starts in $v_k$ and let $g_i$ be the first vertex of $R'$ which lies in $\{g_j:j\geq k\}$.
Then the path $R'g_i$ is a subdigraph of $D_i$ as it contains no vertex $g_j$ with $j>i$.
Hence by \ref{item:length_of_path_to_g}, there is a ray $R(n^*,m^*)$ containing $g_i$ with $|R'g_i|=|R(n^*,m^*)g_i|$.
The two rays $R'$ and
	$R(n^*,m^*)$ are isomorphic as both are isomorphic to $R$, so the two initial
	segments of the same length are isomorphic, and so the tails $g_iR'$ and
	$g_iR(n^*,m^*)$ are also isomorphic. Now assume for a contradiction that
	$g_iR'$ exits $g_iR(n^*,m^*)$ and let $g_{i'}$ be the first vertex where $g_iR'$ exists $g_iR(n^*,m^*)$.
	Note that $g_iR'g_{i'}$ and $g_iR(n^*,m^*)g_{i'}$ are identical.
    From this and the
	isomorphism found before of the two tails beginning at $g_i$, it follows that
	the tails $g_{i'}R'$ and $g_{i'}R(n^*,m^*)$ are also isomorphic. Thus, it follows from
	\ref{item:distance_to_endvertices} that $R'$ has to exit $g_{i'}$ through the segment $g_{i'}R(n^*,m^*)t_1^*$, where
	$t_1^*$ is the end-vertex beyond $g_{i'}$ of the phase containing $g_{i'}$ in $R(n^*,m^*)$,
	because this is the segment that has the correct length of the four segments
	of pairwise distinct lengths incident with $g_{i'}$. This means $g_iR'$ does not exit
	$g_iR(n^*,m^*)$ at $g_{i'}$, a contradiction.

All that remains is to define $(g^0_i,g^1_i)$ for all $i\in\N$.
Suppose that $(g^0_j,g^1_j)$ is already defined for all $j<i$.
We write $R^\eps:=R(n^\eps_i,m^\eps_i)$ for $\eps\in\{0,1\}$.
Our task is to specify suitable vertices $g^\eps_i\in R^\eps$.
Let $x$ be a vertex of $R^0$ that lies beyond all phases of $R^0$ that contain any vertex of $\{ g_0,\ldots,g_{i-1},v_0,\ldots,v_i\}$.
Let $\P$ be the set of all $\{v_0,\ldots,v_i\}$--$x$ paths in $D_{i-1}$ which are isomorphic to initial segments of $R$.
Then none of the vertices on $R^0$ beyond $x$ are identified with other vertices, so
\begin{itemize}
\item[$(*)$] for any $y \in xR^0$, every $\{v_0, \dots , v_k\}$--$y$ path is of the
form $PxR^0y$ for some $P \in \mathcal{P}$.
\end{itemize}
The following claim implies that $\P$ is finite:

\begin{claim}
For all $j\in\N$ and for all vertices $v,w\in D_j$, the set of $v$--$w$ paths in $D_j$ is finite.
\end{claim}

\begin{proof}[Proof of the claim]
We use induction on $j$.
The claim holds for $j = -1$ as $D_{-1}$ is a disjoint union of rays.
Now let $j \geq 0$ and consider an arbitrary $v$--$w$ path $P$ in $D_j$.
If $P$ does not use the vertex $g_j$, then $P$ is also a $v$--$w$ path in $D_{j-1}$, and there are only finitely many such paths by induction.
Otherwise $g_j$ lies on $P$, and $P$ consists of a $v$--$g_j$ path $Q$ concatenated with a $g_j$--$w$ path $Q'$ in $D_j$.
Since $Q$ and $Q'$ are also paths in $D_{j-1}$, there are only finitely many possibilities for $Q$ and $Q'$ by induction.
\end{proof}

We write $\Q$ for the subset of $\P$ consisting of all paths $Q$ with $|Q|\neq|R^0x|$.
Our next step is to find a vertex $z^0$ such that
\begin{itemize}
\item[$(\dagger)$]\label{item:z^0} $z^0$ lies beyond $x$ on $R^0$ and no path from $\Q$ can be extended to a $\{v_0, \ldots, v_i\}$--$z^0$ path in $D_{i-1}$ which is isomorphic to an initial segment of $R$.
\end{itemize}
Let $Q \in \Q$.
Since $|Q| \neq |R^0x|$, then $q \neq x$ where $q$ is the end-vertex (other than the origin of $R^0$) of the initial segment to which $Q$ is isomorphic in $R^0$.
Then by \Cref{lem:unbounded_means_non-isomorphic}, $qR^0$ and $xR^0$ are non-isomorphic. So there is an initial segment of $xR^0$ that is not isomorphic to an initial segment of $qR^0$.
Let $xR^0y_Q$ be such a segment, then $QxR^0y_Q$ is not isomorphic to an initial segment of $R^0$, and hence also not to an initial segment of $R$.
Let $z^0$ be the latest vertex on $R^0$ that is of the form $y_Q$ for some $Q \in \Q$ if $\Q$ is non-empty or let $z^0 := x$ if $\Q$ is empty.
Then by $(*)$, for every $Q \in \Q$ the path $QxR^0z^0$ contains $QxR^0y_Q$ as a subpath.
Thus it follows from the choice of $y_Q$ that $QxR^0z^0$ cannot be isomorphic to an initial segment of $R$.
This shows that our choice of $z^0$ satisfies $(\dagger)$.
Similarly, define a vertex $z^1\in R^1$.

Next, we show that $D_i$ will satisfy \ref{item:length_of_path_to_g} for every choice of vertices $g^\eps_i$ on $R^\eps$ that lie beyond $z^\eps$ on $R^\eps$ for $\eps\in\{0,1\}$.
So suppose we have already fixed vertices $g_i^\eps$ as above and glued them together.
Now consider any $k \leq i$ and any $v_k$--$g_i$ path $S$ in $D_i$ which is isomorphic to an initial segment of $R$.
Then $S$ is also a $v_k$--$g_i^0$ or a $v_k$--$g_i^1$ path in $D_{i-1}$; suppose without loss of generality that the former holds.
Let $v^*$ be the last vertex of $S$ that is contained in the set $\{g_0, \ldots, g_{i-1}, v_0, \ldots, v_i\}$.
Since $z^0$ lies beyond $v^*$ and $g_i^0$ lies beyond $z^0$ on $R^0$, it follows that $S$ must contain $z^0$.
Then $Sx$ is contained in $\P$ but not in $\Q$ by $(\dagger)$ as $Sz^0$ is isomorphic to an initial segment of $R$.
Therefore $|Sx| = |R^0x|$, so $|S| = |R^0g_i|$, which proves \ref{item:length_of_path_to_g}.

Finally, we further specify the choice of $g^0_i$ and $g^1_i$ so that \ref{item:distance_to_endvertices} holds.
Recall that the representing sequence of $R$ is unbounded.
Therefore, we can find a phase $M^0$ of $R^0$ which is contained in $z^0R^0$ and has length at least 3.
We choose $g^0_i$ as an interior vertex of $M^0$ so that $g^0_i$ has different distances to both endvertices of $M^0$.
Next, find a phase $M^1$ of $R^1$ which is contained in $z^1R^1$ such that $|M^1|\geq 2|M^0|+1$.
Then \ref{item:distance_to_endvertices} is fulfilled for a vertex $g^1_i$ in $M^1$ which has distance $|M^0|$ to one endvertex of $M^1$ and hence distance $>|M^0|$ to its other endvertex.
\end{proof}

\medskip
\bibliography{ref}

\end{document}